\documentclass[11pt]{amsart}
\usepackage[a4paper]{anysize}\marginsize{3.0cm}{3.0cm}{1.5cm}{1.5cm}
\pdfpagewidth=\paperwidth \pdfpageheight=\paperheight
\usepackage[utf8]{inputenc}
\usepackage{amsfonts,amssymb,amsthm,amsmath,eucal, graphicx}
\usepackage{verbatim}
\usepackage{color}
\usepackage{fullpage}
\usepackage{hyperref}
\hypersetup{
	colorlinks=true,
	linkcolor=blue,
	filecolor=blue,      
	urlcolor=blue,
	citecolor=blue,
}

\newcommand\call{\mathcal L}

\newcommand\cala{\mathcal A}
\newcommand\calb{\mathcal B}

\newcommand\F{\mathbb F}

\renewcommand\P{\mathbb P}

\DeclareMathOperator{\Sym}{Sym}

\newtheorem{definition}{Definition}[section]
\newtheorem{thm}{Theorem}[section]
\newtheorem{corollary}[thm]{Corollary}
\newtheorem{lemma}[thm]{Lemma}
\newtheorem{problem}[thm]{Problem}
\newtheorem{theorem}[thm]{Theorem}
\newtheorem{remark}[thm]{Remark}
\newtheorem{proposition}[thm]{Proposition}
\newtheorem{thevarthm}[thm]{\varthmname}

\newenvironment{varthm*}[1]{\begin{list}{}{\labelwidth=0cm
\leftmargin=0cm
\listparindent=\parindent}\item[\hspace{\labelsep}\bfseries
#1.]\itshape}{\end{list}}

\title{Line arrangements with many triple points}

\author{Lukas~K\"uhne}
\address{Universit\"at Bielefeld, Fakult\"at f\"ur Mathematik, Bielefeld, Germany}
\email{lkuehne@math.uni-bielefeld.de}

\author{Tomasz Szemberg}
\address{Department of Mathematics, University of National Education Commission, Krakow, Podchorazych 2, PL-30-084 Kra\'ow, Poland}
\email{tomasz.szemberg@gmail.com}

\author{Halszka Tutaj-Gasi\'nska}
\address{Jagiellonian University, Faculty of Mathematics and Computer Science, Lojasiewicza 6,
	PL-30-348 Krako\'ow, Poland}
\email{halszka.tutaj-gasinska@im.uj.edu.pl}

\date{\today}

\begin{document}

\maketitle

\begin{abstract}
In this paper, we construct an infinite series of line arrangements in characteristic two, each featuring only triple intersection points. This finding challenges the existing conjecture that suggests the existence of only a finite number of such arrangements, regardless of the characteristic. Leveraging the theory of matroids and employing computer algebra software, we rigorously examine the existence and non-existence across various characteristics of line arrangements with up to $19$ lines maximizing the number of triple intersection points.
\end{abstract}
\section{Introduction}
A lot of theorems in Euclidean and projective geometry can be expressed as statements on unexpected incidences. This philosophy applies to very basic results in school geometry, e.g., the existence of the orthocenter of a triangle is a statement about three lines (the altitudes) intersecting in one point. The classical Pappus Theorem asserts in turn that some three points are collinear, see \cite{PPG11} for a beautiful and extended overview. This means that the three lines dual to these points intersect in a single point (dual to the line they are contained in).

A lot of interest and work on arrangements with triple points was sparked by one of the problems posed by Jackson in his collection \cite{Jac1821}. In a poetic and pragmatic way, he asked if it is possible to plant nine trees so that in each row containing two of the trees, there must a third tree. The problem in a more formal way was posed by Sylvester \cite{Syl1893} and became known as the \emph{orchard problem}, see \cite{Burr1981} for a nice modern account. The problem was solved by Gallai (Gr\"unwald) in \cite{Gal44}.
\begin{varthm*}{Sylvester-Gallai Theorem}
Let $\call$ be an arrangement of lines in the real projective plane which is not a pencil (i.e. not all lines intersect in a single point). Then there exists a point where exactly two lines from $\call$ intersect (a double point of the arrangement).    
\end{varthm*}
An effective version of the orchard problem asks for a minimal number of double points in dependency on the number $s$ of lines in the arrangement. This problem was solved only recently by Green and Tao \cite{GreTao13}.

From the statement of the Sylvester-Gallai Theorem it becomes immediately clear that incidences in an arrangement of lines depend in a strong way on the underlying field. For example, it follows directly, that the only arrangement with only triple points in the real projective plane is a set of three concurrent lines. Of course this \emph{trivial} arrangement exists in any projective plane.
In the complex projective plane there is the \emph{dual Hesse arrangement} of $9$ lines intersecting altogether in $12$ points, see \cite{LBW19}. Urzua asks in \cite[Question 3]{Urz22} if the dual Hesse arrangement is the only (nontrivial) complex line arrangement with only triple points? Possible approaches to this problem are discussed in \cite[Section 4]{Bounded_and_arrangements} in the context of the bounded negativity conjecture and Harbourne constants. The problems reappear slightly more generally in \cite[Remark 6.13]{CHMN18} and in a recent work by Hanumanthu and Harbourne \cite[Question 11]{HanHar21}.

In positive characteristic there are two more examples. In characteristic $2$, there is the Fano plane $\P^2(\F_2)$ consisting of $7$ lines intersecting in exactly $7$ triple points. In characteristic $3$, there is an example of $9$ lines intersecting in $12$ points. This arrangement results from $13$ lines in the projective plane $\P^2(\F_3)$ after removing one pencil of lines (all $4$ lines passing through a fixed point). No further examples were known up to now. This motivated our research.
Our first result is  the existence of an infinite series of new examples in characteristic $2$, which in particular shows that Conjecture 1.6 in \cite{DHS22} was slightly overoptimistic.

\begin{varthm*}{Theorem A}\label{thm:A}
For any $k\geq 2$ there exists an arrangement of
$$N=2^k+2^{k-1}+\ldots+2+1$$
lines, which intersect in triple points only.
\end{varthm*}

The number of the intersection points is of course
$$\frac{(2^{k+1}-1)(2^{k+1}-2)}{6}.$$
These arrangements are contained in finite projective planes over bigger and bigger fields of characteristic $2$.
Moreover, we found a sporadic arrangement with $19$ lines over $\F_{11}$ which intersects in triple points only.

A related path of research is motivated by the following problem. For an arrangement of lines $\call$, we denote by $t_3(\call)$ the number of points where exactly three lines from $\call$ intersect.
\begin{problem}
Given a positive integer $s$ determine the number $$T_3(s)=\max_{\call} t_3(\call),$$ 
where the maximum is taken over all arrangements $\call$ of $s$ lines, with no restriction on the underlying field.
\end{problem}
These numbers have been determined completely for $s\leq 11$ in \cite{KRtriple}. 

There exists a non-trivial combinatorial upper bound on the numbers $T_3(s)$ introduced by Sch\"onheim in \cite{Sch66}.
\begin{proposition}[Sch\"onheim]\label{thm:Schoenheim}
For $s\geq 1$ let
$$U_3(s)=\left\lfloor\left\lfloor\frac{s-1}{2}\right\rfloor\frac{s}{3}\right\rfloor-\varepsilon(s),$$
where $\varepsilon(s)=1$ if $s\equiv 5\mod 6$ and $\varepsilon(s)=0$ otherwise. Then
$$T_3(s)\leq U_3(s).$$
\end{proposition}
The second result concerns the values of $T_3(s)$ for $12\leq s\leq 19$.

\begin{varthm*}{Theorem B}\label{thm:B}
	We establish the following values for $T_3(s)$ summarized in this table.

	\renewcommand{\arraystretch}{1.5}
	$$\begin{array}{c|cccccccc}
	s & 12 & 13 & 14 & 15 & 16 & 17 & 18 & 19\\
	\hline
	U_3(s)& 20 & 26 & 28 & 35 & 37 & 44 & 48 & 57\\
	\hline
	T_3(s) & 19 & 23  \le 24 & 28 & 35 & 37 & 40\le  & 48 & 57
	\end{array}$$
	In particular we establish the (non)-existence of arrangements whose number of triple points is close to or matches the Schönheim bound.
\begin{itemize}
    \item[i)] There does not exist an arrangement of $12$ lines with $20$ triple points.
    \item[i')] There exists an arrangement of $12$ lines with $19$ triple points.
    \item[ii)] There does not exist an arrangement of $13$ lines with $25$ or $26$ triple points.
    \item[ii')] There exists an arrangement of $13$ lines with $23$ triple points.
    \item[iii)] There exists an arrangement of $14$ lines with $28$ triple points. 
    \item[iv)] There exists an arrangement of $15$ lines with $35$ triple points.
    \item[v)] There exists an arrangement of $16$ lines with $37$ triple points.
    \item[vi)] There exists an arrangement of $17$ lines with $40$ triple points.
    \item[vii)] There exists an arrangement of $18$ lines with $48$ triple points.
    \item[viii)] There exists an arrangement of $19$ lines with $57$ triple points. 
\end{itemize}
\end{varthm*}

Our results mainly rest on two pillars.
First we use an algorithm to compute the realization space of a matroid.
In the language of arrangements, the algorithm checks whether a given geometric lattice is in fact the intersection of some line arrangement over some field.
The details of this procedure are described in Section~\ref{sec:matroids}.
Second, we benefit from the fact that the combinatorial structure of arrangements with only triple points is know as a \emph{Steiner triple system}; a special type of a design which we discuss in Section~\ref{sec:sts}.
Subsequently, we prove Theorem~\hyperref[thm:A]{A} in Section~\ref{sec:char_2} and turn to the proof of Theorem \hyperref[thm:B]{B} in Section~\ref{sec:double_points}.

\subsection*{Open problems}
We conclude the introduction with mentioning a few open problems.
The question of whether there exists an arrangement with only triple points in a field of characteristic~$0$ with more than $9$ lines remains open.
Finally in the realm of Theorem  \hyperref[thm:B]{B}, let us mention that it remains open whether there exists arrangements with
\begin{itemize}
    \item $13$ lines and $24$ triple points and
    \item $17$ lines and at least $41$ triple points.
\end{itemize}

\subsection*{Acknowledgements.}
Our work was initiated during the \textit{Workshop on Complex and Symplectic Curve Configurations} held in Nantes, France, in December 2022. We would like to thank the organizers for stimulating fruitful discussions and providing excellent working conditions.

We would like also to thank Brian Harbourne for helpful comments on an earlier draft. We thank Piotr Pokora for a number of supporting remarks and suggestions while we were working on this project.

K\"uhne was supported by the Deutsche Forschungsgemeinschaft (DFG, German Research Foundation) -- SFB-TRR 358/1 2023 -- 491392403.

Szemberg and Tutaj-Gasi\'nska were partially supported by National Science Centre, Poland, grant 2019/35/B/ST1/00723.

\section{Preliminaries}

In this section we define the objects used throughout this note.

\subsection{Matroids and their realizations}\label{sec:matroids}
Let us start with a gentle introduction to the theory of matroids.
\begin{definition}
A \emph{matroid} $M$ is a pair $(E,\mathcal{B})$, where $E$ is a finite ground set and $\mathcal{B}$ is a non-empty set of subsets of $E$, called \emph{bases}, such that for any two bases $B,B' \in \mathcal{B}$ with some $i \in B\setminus B'$ there exists $j \in B'$ with $(B\setminus \{i\})\cup\{j\} \in \mathcal{B}$.
\end{definition}
The size of $M$ is the cardinality of the ground set $E$ and the common size of all bases is called the rank of $M$. We say that the matroid $M$ is simple if each pair of elements of $E$ is contained in at least one basis.

Let $M = (\{1,\dots, n\}, \mathcal{B})$ be a simple matroid of rank $r$. A \emph{realization} of $M$ over a given field~$\mathbb{K}$ is a matrix $P \in \mathbb{K}^{r\times n}$ such that for all subsets $X\subseteq \{1,\dots,n\}$ of size $r$ we have
\begin{equation}
\label{base}
    {\rm det} \, P_{X} \neq 0 \iff X \in \mathcal{B},
\end{equation}
where $P_{X}$ is the $r\times r$-submatrix consisting of the columns indexed by $X$.
The kernels of the linear forms given by the columns of $P$ define an arrangement $\mathcal{A}$ of $n$ hyperplanes in $\mathbb{K}^{r}$ whose intersection lattice isomorphic to the lattice of flats of matroid $M$. 

The key question for this note is whether a given matroid admits such a realization, that is whether there exists an arrangement of hyperplanes over some field having that matroid as its intersection lattice.
One can answer this question using computer algebra as we discuss now.

Observe that the condition \eqref{base} defines an ideal $I'$ in the ring $R' = R[d]$, where
$$R=\mathbb{Z}[p_{ij} \, : \, i \in \{1, ..., r\}, j \in \{1, ..., n\}]$$
given by
\begin{equation}
\label{ideall}
I' = \langle {\rm det}(P_{N}) \, : \, N \subseteq E \text{ not a basis, } |N| = r \rangle + \bigg\langle 1 - d\prod_{B \in \mathcal{B}(M)} {\rm det}(P_{B})\bigg\rangle < R[d],
\end{equation}
where $P=(p_{ij}) \in R^{r \times n}$ is an $r \times n$ matrix having the variables $p_{ij}$ as entries.
Now the \emph{realization space} of the matroid $M$ is an affine scheme, essentially described by
$$\mathcal{R}(M) :=V(I') \subseteq \mathbb{A}^{rn+1} = {\rm Spec}\, R[d] \rightarrow {\rm Spec}\, \mathbb{Z}.$$
This leads to the following observation.
\begin{proposition}\cite[Thm. 6.8.9]{Oxl11}
The matroid $M$ is realizable over some field if and only if $1 \not\in I'$.
\end{proposition}
This condition can be refined by checking realizability over algebraically closed or finite fields using Gröbner bases.
The details of this algorithm together with a tutorial on how to use it in \texttt{OSCAR} through several examples can be found in~\cite{oscar}.
In particular, this article provides several speed-up techniques to make this computation practically feasible.

\subsection{Steiner triple systems}\label{sec:sts}

The combinatorics underlying arrangements with only triple points is equivalent to Steiner triple systems, a particular type of block designs studied in coding and design theory.
\begin{definition}
	A \emph{Steiner triple system} (STS) is a finite set $S$ together a set $\mathcal{B}$ of $3$-element subsets of $S$ (called blocks) such that every pair of elements of $S$ is contained in exactly one block.
	The order of the Steiner triple system is the cardinality of the set $S$.
\end{definition}

An easy counting argument shows that Steiner triple systems of order $n$ exists if and only if $n$ is of the form $6k+1$ or $6k+3$ for some $k$.
Moreover there is a complete catalog of Steiner triple systems up to permutation of the set $S$ for the orders $7$, $9$, $13$, $15$, and $19$ of which there are $1$, $1$, $2$, $80$, and $11,084,874,829$, respectively. 
See~\cite[Cha.~5]{CD07} for an overview of Steiner triple systems.

Steiner triple systems are combinatorially equivalent to the intersection lattice of arrangements with only triple points as the blocks are then exactly the combinatorial data of the triple intersection points, that is the flats of rank $2$ in the intersection lattice.
In particular, we can use the available database of STSs and check whether the associated matroid is realizable in the sense described above.
We did that for all STSs of order $n\le 15$ and found that of these $84$ designs exactly three are realizable as arrangements.
\begin{enumerate}
	\item The unique STS of order $7$ is the Fano plane and therefore only realizable over fields of characteristic $2$.
	\item The unique STS of order $9$ is the matroid underlying the dual Hesse configuration and therefore realizable over a field $\F$ if and only if there exists a solution to the equation $x^2-x+1$ in $\F$.
	\item Exactly one of the $80$ STSs of order $15$ is realizable. It is however only realizable over fields of characteristic $2$, the smallest being $\F_{16}$.
	Combinatorially it is the truncation of the projective space $\mathbb{P}^3(\F_2)$.
	This is the basis of our general construction underlying Theorem~\hyperref[thm:A]{A} which we describe in the next section.
\end{enumerate}

\section{Arrangements with only triple points in characteristic $2$}\label{sec:char_2}
In this section, we show that truncated projective geometries over $\F_2$ provide examples of line arrangements with only triple points.
The general idea we present here was mentioned before vaguely in the coding theory literature, see for example~\cite{Limbos}.
To the best of our knowledge we give the first complete proof of this statement.
By the way of illustrating the approach we begin with the simplest example and then pass to the general case.

Let $X=\P^3(\F_2)$ be the projective $3$-space over $\F_2$ considered as a subset of $Y=\P^3(\F_{2^4})$. Let $Z$ be the union of planes in $Y$, which are generated by triples of non-collinear points in $X$.
\begin{lemma}\label{lem: existence of the projection centre}
The set $W:=Y\setminus Z$ is non-empty.
\end{lemma}
\begin{proof}
Here, we perform a very rough count of points. Any three non-collinear points in $X$ span a plane in $X$ and also a plane in $Y$, which contains the ''smaller'' plane in $X$. A plane in $X$ contains $7=2^3-1$ points, as it is just a Fano plane. A plane in $Y$ contains $273=16^2+16+1$ points. There are $15=2^4-1$ planes in $X$ (by duality there are as many hyperplanes as there are points). They generate $15$ planes in $Y$, which altogether contain no more than $4 095= 15\cdot 273$ points. Since there are $4 369=16^3+16^2+16+1$ points in $Y$, the claim follows.
\end{proof}
Let us consider a projection $\pi$ from a point $P\in W$ to a hyperplane in $Y$. The crucial observation here is that $P$ does not lie on any line passing through $2$ or more points in $X$. Indeed, if it were contained in such a line (extended to $Y$), it would also be contained in one of the planes building the set $Z$ in Lemma \ref{lem: existence of the projection centre}, but this would contradict its choice not in $Z$. Hence $\pi$ restricted to $X$ is injective. We may think of $\pi$ as ''spreading'' the $3$-dimensional projective space $X$ on the projective plane defined over the extension field $\F_{2^4}/\F_2$.

Our next Lemma is a simple count of lines in $X$.
\begin{lemma}\label{lem: number of lines}
There are $35$ lines in $\P^3(\F_2)$.
\end{lemma}
\begin{proof}
As we already observed there are $15$ points in $\P^3(\F_2)$. Every line is determined by a pair of distinct points and since there are $3$ points on each line we have
$$\mbox{number of lines in } \P^3(\F_2)\;=\; \frac{1}{3}\binom{15}{2}=35$$
as asserted.
\end{proof}

Since collinearities remain unaltered under projections, there are $35$ triples of collinear points in $X'=\pi(X)$. In the dual projective plane $H^*=\P^2(\F_{2^4})^*$ these triples of collinear points become triples of lines meeting in a point. There are altogether $15$ points in $X'$, so they correspond to $15$ lines in $H^*$. These lines intersect by $3$ in all points which are dual to the $35$ lines determined by points in $X'$. But then elementary combinatorics shows that there are no other intersection points among the $15$ lines.
\begin{corollary}\label{cor:15 lines with 35 triple points exist}
There exists  an arrangement of $15$ lines in $\P^2(\F_{2^4})$ with $35$ triple intersection points and no other intersection points.
\end{corollary}

The argument above is sharp with respect to the size of the field; we determined with \texttt{OSCAR} that this matroid is not realizable over $\F_2$, $\F_4$, and $\F_8$.

In general for $n\ge 3$, we consider the lines in the projective geometry $\P^n(\F_2)$.
There are $2^{n+1}-1$ such lines that intersect in $\frac{1}{6}(2^{n+1}-1)(2^{n+1}-2)$ many points.
Crucially every such intersection point is a triple point and there are no other intersection points.

The point of the argument below is that we can find the same configuration of lines in a projective \emph{plane} of a field of characteristic two of larger size.
So these are all line arrangements of size $2^{n+1}-1$ with only triple points in larger and larger fields of characteristic two.
Note that combinatorially the matroids are the truncations of the projective geometries $\P^n(\F_2)$ to rank $3$.

The general result goes along the same lines, but we are not able to find the smallest possible $m$, such that the configuration of $2^{n+1}-1$ lines with only triple intersection points  exists in 
$\F_{2^m}$. However, we are able to show:

\begin{theorem}
Take an integer $k\geq 4$. Take $m=3k-4$. Then
 there exists an arrangement of $2^{k+1}-1$ lines in $\P^2(\F_{2^{m}})$  with $\frac{(2^{k+1}-1)(2^k-1)}{3}$ triple intersection points and no other intersection points.
\end{theorem}
\begin{proof}

The idea of the proof, based on \cite{Limbos}, is to take the configuration $C$ of lines in $\P^k(\F_2)$.
We embed this space into $\P^k(\F_{2^m})$ and project the configuration to $\P^{k-1}(\F_{2^m})$ from a point such that the collinearities are preserved.
Then we repeat this process until we arrive at a configuration in $\P^2(\F_{2^m})$.
As each  line from $C$ contains 3 points from $\F_2$, the dual  configuration is the desired one.

The condition on a projection to preserve collinearities in  the configuration $C$ is that the center of the projection does not lie on any plane containing 
at least two lines from~$C$.
So we check when it is possible to make such a choice. 

The number $n_3(C)$  of planes in $\P^k(\F_{2^m})$ containing at least two lines from $C$ is at most
$$n_3(C)=\frac{\textrm{number of choices of three points from $\P^k(\F_2)$}}{\textrm{number of choices of three points on a plane in $\P^2(\F_2)$ }}$$
 so
 $$n_3(C)=\frac{\binom{2^{k+1}-1}{3}}{\binom{7}{3}}=\frac{(2^{k+1}-1)(2^{k+1}-2)(2^{k+1}-3)}{6\cdot 35}.$$

The number of points in $\P^k(\F_{2^m})$ on each such plane is $p:=(2^m)^2+2^m+1.$

We need to find a center of projection outside the planes in every $\P^j(\F_{2^m})$ for $j=3,4,\dots,k$. Thus, for each $j=3,4,\cdots k$ the number of points in  $\P^j(\F_{2^m})$, should be greater than 
$p\cdot n_3(C)$.

This means, that  we are able to find the desired center of projection if
$$p\cdot n_3(C)<(2^m)^3+(2^m)^2+2^m+1=\frac{(2^m)^4-1}{2^m-1}.$$
This yields
$$\frac{(2^m)^3-1}{2^m-1}\frac{\binom{2^{k+1}-1}{3}}{\binom{7}{3}}<\frac{(2^m)^4-1}{2^m-1},$$
which simplifies to
$$0<-(2^{3m} - 1)(2^{k+1}-1)(2^{k+1}-2)(2^{k+1}-3) + 
 210(2^{4m}-1).$$

For $m=3k-4$ this inequality is satisfied, because then we get
 $$0<-216+11\cdot2^{k+1}-3\cdot2^{2k+3}+3\cdot2^{9k-11}-11\cdot2^{10k-11}+3\cdot2^{11k-9}+2^{3k+3}+41\cdot2^{12k-15},$$
 which is true for every $k\ge 4$.
\end{proof}

This construction in fact generalizes to every finite field $\F_{q}$ which yields line arrangements over field extensions of $\F_q$ with all intersection points of size $q+1$.
\begin{remark}\label{rem:char q}
Let $k\geq 3$ be an integer, let $q>2$ be (a power of) a prime number. A reasoning similar to the one above yields a configuration of $q^{k+1}-1$ lines in $\P^2(\F_{q^{3k-5}})$ with $\frac{(q^{k+1}-1)(q^k-1)}{(q-1)(q+1)}$ points which are $(q+1)$-fold.
\end{remark}
\begin{proof}
Now take as the set $C$ all  lines in $\P^k(\F_q)$. Embed $\P^k(\F_q)$ in 
$\P^k(\F_{q^{3k-5}})$. We have to ensure the existence of a suitable center of the collinearities-preserving projection from $\P^k(\F_{q^{3k-5}})$ to $\P^{k-1}(\F_{q^{3k-5}})$, then from
    $\P^{k-1}(\F_{q^{3k-5}})$ to $\P^{k-2}(\F_{q^{3k-5}})$ etc.

Again, the projection preserves collinearities in the lines from $C$ if the center of the projection is not on any plane with at least two lines from $C$, thus again we get that the number of forbidden  points is at most 

$$((q^{3k-5})^2+q^{3k-5}+1)\cdot\frac{\binom{\frac{q^{k+1}-1}{q-1}}{3}}{\binom{q^2+q+1}{3}}.$$

After all but one projections, this number should be strictly less than the number of points in $\P^3(\F_{q^{3k-5}})$, so we get the inequality:

$$((q^{3k-5})^2+q^{3k-5}+1)\cdot\frac{\binom{\frac{q^{k+1}-1}{q-1}}{3}}{\binom{q^2+q+1}{3}}<(q^{3k-5})^3+(q^{3k-5})^2+q^{3k-5}+1.$$
This is equivalent to:
$$
0<-q^2+2q^3+2q^4
-3q^6+q^8 
+q^k-2q^{1 + k}-2q^{2 + k} + 3 q^{2 + 2 k}+ $$ 
$$+2q^{-2 + 3 k}+ 
 2 q^{-1 + 3 k}
 -3 q^{1 + 3 k} -q^{2+3k}++q^{3+3k}+$$
 $$
 +q^{-5+4k}-2q^{-4+4k}-2q^{-3+4k}+3q{-5+5k}+$$
 $$-q^{-3+6k}+q^{-8+6k}+2q^{-7+6k}+2q^{-6+6k}-3q^{-4+6k}+q^{-2+6k}+q^{-10+7k}-2q^{-9+7k}-2q^{-8+7k}+3q^{-8+8k}+$$
 $$-q^{-8+9k}-2q^{-14+9k}-q^{-13+9k}+2q^{-12+9k}+2q^{-11+9k}-3q^{-9+9k}+q^{-7+9k}.$$

Using the fact that $q\geq 3$ and $k\geq 3$ this is elementary (however tiresome) to show that this inequality holds. 
\end{proof}

\section{Line arrangements with the maximal number of triple points}\label{sec:double_points}

We begin with the following simple but useful observation.
\begin{lemma}\label{lem:steiner minus line}
    Let $\cala$ be a realizable Steiner triple system on $s$ lines. Then 
    $$T_3(s-1)=U_3(s-1),$$
    i.e., there exists an arrangement of $s-1$ lines with the number of triple points attaining the Sch\"onheim bound.
\end{lemma}
\begin{proof}
   For $s$ we have
$$\left\lfloor\left\lfloor\frac{s-1}{2}\right\rfloor\frac{s}{3}\right\rfloor=\frac{s(s-1)}{6},$$
hence $s=6k+1$ or $s=6k+3$ for a positive integer $k$ (we omit the trivial case $s=3$).

\noindent
It follows that for $s=6k+1$ we have $s-1=6k$ and
\begin{equation}\label{eq:s=6k+1}
    U_3(s-1)=\left\lfloor\left\lfloor\frac{6k-1}{2}\right\rfloor\frac{6k}{3}\right\rfloor=(3k-1)(2k)=6k^2-2k.
\end{equation}
On the other hand in $\cala$ there are exactly $\frac{s-1}{2}$ triple points on each line and $6k^2+k$ triple points altogether. Let $\calb$ be an arrangement obtained from $\cala$ by removing one line. Then the number of triple points is
$$(6k^2+k)-3k=6k^2-2k,$$
which shows that the bound in \eqref{eq:s=6k+1} is attained.

\noindent
For $s=6k+3$ we have $s-1=6k+2$ and
\begin{equation}\label{eq:s=6k+3}
    U_3(s-1)=\left\lfloor\left\lfloor\frac{6k+1}{2}\right\rfloor\frac{6k+2}{3}\right\rfloor=k(6k+2)=6k^2+2k.
\end{equation}
In $\cala$ there are this time altogether $6k^2+5k+1$ triple points and $3k+1$ points on each line. Removing as above one line from $\cala$ we obtain an arrangement $\calb$ with 
$$(6k^2+5k+1)-(3k+1)=6k^2+2k$$
triple points. Again, the bound in \eqref{eq:s=6k+3} is attained.
\end{proof}
Here we want to provide a computational proof of nonexistence (over any field) of configurations of $12$ lines with $20$ triple (and $6$ double points) and $13$ lines with $26$ triple points.

\begin{proof}[Proof of Theorem~B]
For \textbf{(i)} note that a line arrangement with $12$ lines and $20$ triple points either has one additional quadruple point or $6$ additional double points.
Consulting the database of all matroids of rank $3$ and $12$ elements yields that there is no matroid of the former and exactly $5$ matroids of the latter type.
We checked that these five matroids are all not realizable over any field (four of these matroids even contain the non-Pappus matroid which immediately implies that they are not realizable).
Thus, there is no line arrangement of $12$ lines with $20$ triple points over any field.
\medskip

\noindent
The arrangements announced in \textbf{(i')} are well known, see for example \cite{FKLBTG19} for a description of two such arrangements and their role the containment problem in commutative algebra.
\medskip

\noindent
As for \textbf{(ii)} the combinatorial structure of a line arrangement with $13$ lines and $26$ triple points is a Steiner triple system on $13$ elements.
As mentioned above, up to permutation of the elements, there are exactly two such designs~\cite{CD07}:
    \begin{footnotesize}
    \begin{align*}
       \bullet & [1, 2, 3], [1, 4, 5], [1, 6, 7], [1, 8, 9], [1, 10, 11], [1, 12, 13], [2, 4, 6], [2, 5, 8], [2, 7, 9],\\& [2, 10, 12], [2, 11, 13], 
        [3, 4, 10], [3, 5, 6],
         [3, 7, 11], [3, 8, 13], [3, 9, 12], [4, 7, 12], [4, 8, 11], \\&[4, 9, 13], [5, 7, 13], [5, 9, 10], [5, 11, 12], [6, 8, 12], [6, 9, 11], [6, 10, 13], [7, 8, 10].\\
         \bullet &[1, 2, 3], [1, 4, 5], [1, 6, 7], [1, 8, 9], [1, 10, 11], [1, 12, 13], [2, 4, 6], [2, 5, 8], [2, 7, 9],\\& [2, 10, 12], [2, 11, 13], [3, 4, 10], [3, 5, 6], [3, 7, 11], [3, 8, 12], [3, 9, 13], [4, 7, 12], [4, 8, 13],\\& [4, 9, 11], [5, 7, 13], [5, 9, 10], [5, 11, 12], [6, 8, 11], [6, 9, 12], [6, 10, 13], [7, 8, 10].
    \end{align*}
    \end{footnotesize}
    Limbos already checked that both matroids are not realizable over any field (``projectively embeddable'' in their language)~\cite{Limbos}.
    We confirmed this observation with our realization algorithm.
    Hence, there is no line arrangement with this underlying matroid.
\smallskip

Suppose there exists an arrangement with $13$ lines and $25$ triple points.
There must be one line that contains at most $5$ triple points, as we would have 26 triple points otherwise in total.
Removing this line therefore yields and arrangement with $12$ lines and at least $20$ triple points
By \textbf{(i)} such an arrangement does not exists and thus also the arrangement with $25$ triple points can not exists.

\noindent
The arrangement in \textbf{(ii')} arises by removing three triple points from one of the STSs of order $13$ and replacing them by three double intersection points each.
Computations in \texttt{OSCAR} revealed that the resulting matroid is only realizable in a field of characteristic $7$.
One such arrangement with $13$ lines and $23$ triple points is defined via the following matrix in $\F_7$.
\[
\begin{pmatrix}
1 &  1&   1&   1&   1&   0&   1&  0&   1&   1&   1&   0&   1\\
1&   4&   3&   0&   5&   1&   0&   1&   3&   4&   5&   0&   1\\
1&   6&   4&   2&   6&   1&   0&   0&   1&   0&   2&   1&   4
\end{pmatrix}.
\]
The existence of an arrangement with $13$ lines and $24$ triple points remains open for now.
\smallskip

\noindent
The Steiner triple system announced in \textbf{(iv)} exists in characteristic $2$ by Corollary \ref{cor:15 lines with 35 triple points exist}, hence also the arrangement in \textbf{(iii)} exists by Lemma~\ref{lem:steiner minus line}.
\smallskip

\noindent
\textbf{(v)}
The existence of an arrangement of $16$ lines with $37$ triple points is claimed in \cite[Theorem 2]{BGS_Orchard_1974}. The authors use certain $15$-division points on an elliptic curve to construct (after dualizing) an arrangement of $15$ lines with $31$ triple points and claim that the parameter $\alpha$ in their construction can be chosen in such a way that six out of twelve double points of the configuration are collinear. However the value of $\alpha$ is not specified and checking the construction directly requires knowledge of coordinates of $15$-division points on a specific cubic. Our approach, motivated by \cite{BGS_Orchard_1974} is much more explicit.

The matroid we want to realize geometrically is determined by the following triples (of lines intersecting in one point):
$$\begin{array}{cccccccc}
1, 8, 9, & 1, 7, 10, & 1, 6, 11, & 1, 5, 12, & 1, 4, 13, & 1, 3, 14, & 1, 2, 15, & 2, 7, 9,\\
2, 6, 10, & 2, 5, 11, & 2, 4, 12, & 2, 3, 13, & 2, 8, 16, & 3, 7, 8, & 3, 6, 9, & 3, 5, 10,\\
3, 4, 11, & 3, 12, 16, & 4, 6, 8, & 4, 5, 9, & 4, 14, 15, & 4, 10, 16, & 5, 6, 7, & 5, 13, 15,\\
5, 14, 16, & 6, 13, 14, & 6, 12, 15, & 7, 12, 14, & 7, 11, 15, & 7, 13, 16, & 8, 12, 13, & 8, 11, 14,\\
8, 10, 15, & 9, 11, 13, & 9, 10, 14, & 9, 15, 16, & 10, 11, 12. & & &
\end{array}
$$
We checked with \texttt{OSCAR} that it can be realized over the reals by lines given by equations $\alpha x+\beta y+\gamma z=0$ whose coefficients $(\alpha, \beta, \gamma)$ are given by the rows of the following matrix
\begin{footnotesize}
$$
\begin{pmatrix}
1 & 0 & 0 \\
0 & 1 & 0 \\
0 & 0 & 1 \\
1 & 1 & 1 \\
1 & 1 & 2 \\
0 & 1 & 1 \\
1 & -2\varepsilon + 3 & 2 \\
1 & 2\varepsilon - 2 & 4\varepsilon - 4 \\
1 & -2\varepsilon^2 + 3\varepsilon & 2\varepsilon - 1 \\
1 & -2\varepsilon^2 + 3\varepsilon & \varepsilon \\
1 & 2\varepsilon - 2 & 2\varepsilon - 1 \\
1 & -2\varepsilon + 3 & 4\varepsilon - 4 \\
1 & -2\varepsilon^2 + 2\varepsilon + 1 & 1 \\
1 & 0 & 2\varepsilon-2 \\
1 & -2\varepsilon + 1 & 0 \\
1 & -2\varepsilon^2 + 2\varepsilon + 1 & \varepsilon
\end{pmatrix},
$$
\end{footnotesize}
where $\varepsilon$ satisfies the equation $4\varepsilon^2-6\varepsilon+1=0$.

A nice realization of the same matroid over $\F_{11}$ is given by the same token by the columns of the following matrix:
\begin{footnotesize}
$$
\begin{pmatrix}
1&   0&   0&   1&   1&   1&   1&    1&   1&   1&   1&   1&   0&   1&   1&    1\\
0&   1&   0&   1&   5&   7&   6&    6&   7&   5&   1&   8&   1&   0&   3&    8\\
0&   0&   1&   1&   2&   3&   8&   10&   8&   3&   2&   1&   1&   7&   0&   10    
\end{pmatrix}.$$
\end{footnotesize}
\smallskip

\noindent
We now turn to the case of $17$ lines in \textbf{(vi)}. The existence of an arrangement with $40$ triple points follows from the construction in \cite{BGS_Orchard_1974}. It is not known whether there are arrangements with more than $40$ triple points.
\smallskip

\noindent
We describe an arrangement of $19$ lines with only triple points in $\F_{11}$ in Section~\ref{sec:sporadic}.
This proves the claimed statement in \textbf{(viii)}.
Hence also the arrangement in \textbf{(vii)} exists by Lemma~\ref{lem:steiner minus line}.
There also exists a combinatorially different arrangement of $18$ lines with $48$ triple points in characteristic zero described by Tombarkiewicz and Zi\c eba in \cite{TomZie21}.
\end{proof}
    
\subsection{Sporadic example}\label{sec:sporadic}

The automorphism group of a matroid $M=([n],\mathcal{B})$ is the group of the permutations $\sigma\in \Sym([n])$ that preserve the bases $\mathcal{B}$, that is $B\in\mathcal{B}$ if and only if $\sigma(B)\in \mathcal{B}$.

We observed that matroids underlying the line arrangements with only triple points all have a relatively large automorphism group.
The matroid of the dual Hesse arrangement has an automorphism group of size $432$.
The matroids of the family of line arrangements in characteristic $2$ in Section~\ref{sec:char_2} have the automorphism group $PSL_n(2)$ with
\[
|PSL_n(2)|=\prod_{i=0}^{n-1}(2^n-2^i).
\]

We therefore studied Steiner triple systems with a relatively large group of automorphisms.
All Steiner triple systems of order $19$ and $21$ with a non-trivial automorphism group were computed in~\cite{sts_symmetric_19} and \cite{symmetric_sts_21}, respectively.

There are exactly $104$ Steiner triple systems of order $19$ with an automorphism group of order at least $9$, and these are given in the appendix to~\cite{sts_symmetric_19}.
We checked the realizability of the corresponding matroids and somewhat surprisingly, exactly one of those matroids is realizable.
The matroid has an automorphism group of order $57$ and it is only realizable over $\F_{11}$.
It is in fact a protectively unique arrangement and one realization is given by this matrix over $\F_{11}$:
\[
\begin{pmatrix}
1 &  0 &  0 & 1 &  1 & 0 &  1 & 1 & 1 &  1 &  1 &  1 & 1 &  1 &  1 & 1 &  1 & 1 &  1\\
0  & 1 &  0 &  1 &  0 &  1 &  1 &  10 &  2 &   3 &   4 &   8 &  10 &   6 & 8 & 2 &   6 &  3 & 4\\
0   & 0 &  1&   0 &  1  & 5  & 6  &  1 &  10 &  7 &  6 &  3 &    4 &  2 & 10 &   7&   4&  2 &  3 
\end{pmatrix}.
\]

So this is a line arrangement over $\F_{11}$ with $19$ lines and $57$ triple points.
We verified that the automorphism group $G_{57}$ acts transitively on the triple points, so we have a $G_{57}$-orbit of triple intersection points. Moreover, we found that there exists another Steiner triple system of order $19$ with automorphism group $G_{57}$ of order $57$, but it cannot be realized over any field, which makes our sporadic example even more intriguing. 

Subsequently, we consider realizability of all Steiner triple systems of order $21$ with an automorphism group of size at least $7$ enumerated in~\cite{symmetric_sts_21}.
All corresponding matroids were however not realizable over any field.


\end{document}